\documentclass[reqno]{amsart}

\usepackage{amsfonts}
\usepackage{amsthm}
\usepackage{amstext}
\usepackage{amsmath}
\usepackage{amscd}
\usepackage{amssymb}
\usepackage[mathscr]{eucal}
\usepackage{epsf}
\usepackage{array}
\usepackage{url}
\usepackage{enumerate}
\usepackage{enumitem}
\usepackage{graphicx}
\usepackage{subcaption} 
\usepackage{accents}
\RequirePackage[numbers]{natbib}

\newcommand{\deter}{\textup{DET}}
\newcommand{\rr}{\textup{RR}}
\newcommand{\corsum}{\textup{C}}
\newcommand{\corint}{\textup{c}}

 \makeatletter
 \def\@seccntformat#1{\csname the#1\endcsname.\quad}
 \makeatother

\theoremstyle{plain}
\newtheorem{theorem}{Theorem}[section]
\newtheorem{proposition}[theorem]{Proposition}
\newtheorem{corollary}[theorem]{Corollary}
\newtheorem{lemma}[theorem]{Lemma}

\theoremstyle{definition}

\theoremstyle{remark}
\newtheorem{remark}[theorem]{Remark}

\theoremstyle{remark}

\numberwithin{equation}{section}

\newcommand{\NNN}{\mathbb{N}}

\newcommand{\length}[1]{\lvert #1 \rvert}
\newcommand{\uu}{\omega}  
\newcommand{\AAa}{A}      
\newcommand{\eps}{\varepsilon}

\begin{document}
\bibliographystyle{abbrv}

\title[Complexity and invariant measure of the period-doubling subshift]
 {Complexity and invariant measure \\of the period-doubling subshift
 }

\author[M. Pol\'akov\'a]{Miroslava Pol\'akov\'a}
\address{
  Department of~Mathematics, Faculty of~Natural Sciences, Matej Bel University, Tajovsk\'eho~40,
  Bansk\'a Bystrica, Slovakia
}
\email{miroslava.sartorisova@umb.sk}

%


\begin{abstract}
Explicit formulas for complexity and unique invariant measure
of the period-doubling subshift can be derived from those for the Thue-Morse subshift,
obtained by Brlek, De Luca and Varricchio, and Dekking.
In this note we give direct proofs
based on combinatorial properties of the
period-doubling sequence.
We also derive explicit formulas for correlation integral
and other recurrence characteristics of the period-doubling subshift.
As a corollary we obtain that the determinism of this subshift
converges to $1$ as the distance threshold approaches $0$.

\end{abstract}

\maketitle

\thispagestyle{empty}

\section{Introduction}\label{S:intro}

The period-doubling sequence $$\omega = \omega_1 \omega_2 \ldots = 0100 \, 0101 \, 0100 \, 0100 \ldots$$
can be defined in various ways.
First, its $n$-th member is $0$ if and only if the largest $k$ such that
$k$-th power of $2$ divides $n$, is odd; otherwise it is $1$.
Second, $\omega$ is a unique fixed point of the primitive
substitution $0 \mapsto 01$, $1 \mapsto 00$.
Third, $\omega$ is the Toeplitz sequence defined by patterns $(0*)$ and $(1*)$;
for the general definition of Toeplitz sequences
see \cite{jacobs19690, downarowicz2005survey}.

The induced subshift, again called period-doubling, is strictly ergodic
(i.e. it is minimal and has a unique invariant measure) and
has zero topological entropy.
Dynamical properties of this subshift were studied already in 50s and 60s, see
the book \cite{gottschalk1955topological} by Gottschallk and Hedlund
and the article \cite{jacobs19690} by Jacobs and Keane;
for some recent references see e.~g.~\cite{damanik2000local, avgustinovich2006sequences, coven2008characterization}.
In the book \cite{kurka2003topological},
period-doubling subshift (called Feigenbaum subshift therein)
is mentioned many times as an example with interesting dynamics.

The period doubling sequence is tightly connected with the Thue-Morse sequence,
which is a unique fixed point of the primitive substitution
$0 \mapsto 01, 1 \mapsto 10$ which starts with $0$.
Complexity of this sequence was studied in
\cite{brlek1989enumeration, de1989some} and
the invariant measure was considered in \cite{dekking1992thue}.

The period-doubling sequence $\omega$
is a 2-to-1 image of the Thue-Morse sequence \cite[Definition 12.51]{gottschalk1955topological};
every subword $w=w_1 \ldots w_n$ of $\omega$ corresponds to exactly two subwords $u = u_1 \ldots u_{n+1}$
of the Thue-Morse sequence such that $u_i = u_{i+1}$ if and only if $w_i = 1$.
This relation and the results from \cite{brlek1989enumeration, de1989some, dekking1992thue}  yield
formula \eqref{EQ:complexity}  for the complexity of $\omega$,
and a description of the unique invariant period-doubling measure $\mu$;
namely for every allowed $m$-word $u$ $(m \geq 1)$ we have
$$ \mu ([u]) = \frac{2}{3\cdot 2^k} \qquad \text{or} \qquad \mu ([u]) = \frac{1}{3\cdot 2^k}, $$
where $k \geq 0$ is such that $2^k \leq m < 2^{k+1}$.

These results are well-known, but cannot be easily found in the literature.
Since the period-doubling substitution is of constant length, it is possible to study the complexity of it using a general method from \cite{Mosse96}; however, it yields a set of non-trivial recurrent formulas and it seems difficult to derive \eqref{EQ:complexity} from them.

Dekking \cite{dekking1992thue} has described factor frequencies
in the Thue-Morse sequence and the Fibonacci sequence.
Factor frequencies in generalized Thue-Morse words were studied
in \cite{balkova2011factor}.
Frid \cite{frid1998frequency} has obtained
a precise description of factor frequencies
in a wide class of fixed points of substitutions
(the so-called circular marked uniform substitutions, for definitions see \cite{frid1998frequency})
including the Thue-Morse sequence,
but the period-doubling sequence, being not marked, does not belong to this class.

Here we give a direct proof of formula \eqref{EQ:complexity}
based on the combinatorics of the period-doubling sequence $\omega$,
and we derive some other properties of $\omega$.
One of them states that if the length $m$ is a power of $2$,
then the set of all $m$-words is equal to the set of first $(3/2)m$ subwords of $\omega$.

\begin{theorem}[Complexity of the period-doubling sequence]\label{T:complexity}
	Let $m\in\NNN$ be arbitrary. Then the number of $m$-words in the period-doubling
	sequence is given by
	\begin{equation}
	\label{EQ:complexity}
	p(m) = 
	\begin{cases}
	2                      &\text{if } k=0;
	\\
	3\cdot 2^{k-1}+2q      &\text{if } k\ge 1 \text{ and }\ q\le 2^{k-1};
	\\
	4\cdot 2^{k-1}+q       &\text{if } k\ge 1 \text{ and }\ q > 2^{k-1};
	\end{cases}
	\end{equation}
	where $k\ge 0$ and $0\le q<2^k$ are such that $m=2^k+q$. 
	
	Furthermore, for $m=2^k \geq 2$, the set of all $m$-words is
	$$\mathcal{L}_m (\uu) = \left\{w_{i}^{(m)} \colon 1 \leq i \leq \frac{3}{2}m \right\},$$
	where $w_i^{(m)} = \omega_i \ldots \omega_{i+m-1}$.
	
\end{theorem}

Further, we can say exactly what is the measure of a given cylinder.

\begin{theorem}\label{T:freq-m=general}
	Let $\mu$ be the unique invariant measure of the period-doubling subshift.
	Let $u$ be an allowed $m$-word $(m \geq 1)$, $k \geq 0$ be such that $2^k \leq m < 2^{k+1}$
	and $i$ be the least integer such that $u = w_i^{(m)}$. Then  $1 \leq i \leq 3 \cdot 2^k$, and
	
	 \begin{enumerate}
		\item if $i \leq 2^k-q$, or $q < 2^{k-1}$ and $2^k < i \leq 2^k + 2^{k-1}-q$, then $\mu \left( [u] \right) = 2/(3 \cdot 2^k)$; \label{C1:T:freq-m=general}
		\item otherwise $\mu \left( [u] \right) = 1/(3 \cdot 2^k)$. \label{C2:T:freq-m=general}
	\end{enumerate}  

\end{theorem}

\begin{corollary}
	\label{COR:number-of-mu}
	Let $m=2^k+q$ with $k\geq 0$ and $0\leq q < 2^k$.
	Denote by $r(m)$ the number of $m$-words $u$ such that $\mu([u]) = 2/(3\cdot 2^k)$. Then
	\begin{equation*}
		r(m) = 
		\begin{cases}
			1                     &\text{if } k=0;
			\\
			3\cdot 2^{k-1}-2q      &\text{if } k\ge 1 \text{ and }\ q < 2^{k-1};
			\\
			2^{k}-q       &\text{if } k\ge 1 \text{ and }\ q \geq 2^{k-1}.
		\end{cases}
		\end{equation*}
	
\end{corollary}

Precise knowledge of the invariant measure $\mu$ allows us to derive formulas
for correlation integrals (for corresponding definitions see Section~\ref{S:prelim}).
For $\eps >0$ define $m_\eps \in \NNN$ as follows:
if $\eps \geq 1$ then $m_\eps = 0$; otherwise $m_\eps$ is a unique positive integer such that
\begin{equation} \label{EQ:m-eps}
	2^{-m_\eps} \leq \eps < 2^{-m_\eps+1}. 
\end{equation}

\begin{theorem}\label{THM:correl-integ-introd}  
	Let $\eps > 0$. Then the correlation integral of the unique invariant measure $\mu$ of the period-doubling subshift is
	\begin{equation*}
	\corint(\mu, \eps) = 
	\lim\limits_{n \to \infty} \corsum(\omega, n, \eps) = 
	\begin{cases}
	1 & \text{  if } m_\eps = 0; \\
	5/9  & \text{  if } m_\eps = 1; \\
	(3\cdot 2^{k+1}-4q)/((3\cdot 2^k)^2)    & \text{  if }m_\eps \geq 2 \text{ and }q < 2^{k-1} ;\\
	(5 \cdot 2^{k}-2q)/((3\cdot 2^k)^2) & \text{  if }m_\eps \geq 2 \text{ and }q \geq 2^{k-1};
	\end{cases}
	\end{equation*}
	where $k \geq 0$ and $0 \leq q < 2^{k}$ are integers such that $m_\eps=2^k+q$.
\end{theorem}

For simple inequalities for $\corint(\mu, \eps)$ see Corollary~\ref{C:ceps-limits}.
Theorem~\ref{THM:correl-integ-introd} together with the results from \cite{grendar2013strong}
yield asymptotic values for two of the basic measures
of recurrence quantification analysis: recurrence rate ($\rr$) and determinism ($\deter$).

\begin{theorem}[Recurrence rate of $\omega$]\label{THM:RR}  
	Let $\ell \geq 1$ and $\eps > 0$.
	Then the recurrence rate $\rr_\ell(\omega, \eps)$ exists and
	\begin{equation*}
	\rr_\ell(\omega, \eps) = 
	\begin{cases}
	 1 & \text{  if } m_\eps = 0; \\
	 5/9 & \text{  if } m_\eps = 1 \text{ and } \ell = 1; \\
	(3\cdot 2^{k+1}-4q+4\ell-4)/((3\cdot 2^k)^2)    & \text{  if }m_\eps + \ell \geq 3 \text{   and }q < 2^{k-1} ;\\
	(5 \cdot 2^{k}-2q+2\ell-2)/((3\cdot 2^k)^2) & \text{  if } m_\eps + \ell \geq 3\text{ and }q \geq 2^{k-1};
	\end{cases}
	\end{equation*}
	there, for $m_\eps + \ell \geq 3$, $k \geq 1$ and $0 \leq q < 2^k$ are unique integers
	such that $m_\eps + \ell -1 = 2^k+q$.
	
\end{theorem}

\begin{theorem}[Determinism of $\omega$] \label{THM:DET=1}
	Let $\ell \geq 2$ and $\eps > 0$. Then
	$\deter_\ell(\omega, \eps)$ exists,
	$$\deter_\ell(\omega, \eps)=
	\frac{\rr_\ell(\omega, \eps)}{\rr_1(\omega, \eps)}$$
	and
	$$ \lim\limits_{\eps \to 0} \deter_\ell(\omega, \eps) = 1.$$
	Moreover, $\deter_\ell(\omega, \eps) = 1 $ if and only if one of the following three cases happens:
	\begin{enumerate}[label=(\alph*)]
		\item \label{Case1-in-THM:DET=1}
		$\eps \geq 1$;
		
		\item \label{Case2-in-THM:DET=1}
		$2^k \leq m_\eps < m_\eps + \ell - 1 < 2^k+2^{k-1}$ 
		for some $k \in \NNN$;
		
		\item \label{Case3-in-THM:DET=1}
		$2^k+2^{k-1} \leq m_\eps < m_\eps + \ell - 1 < 2^{k+1}$
		for some $k \in \NNN$.
		
	\end{enumerate}

\end{theorem}
	
Figure~\ref{FIG:rr2,det2} illustrates $\rr_2$ and $\deter_2$ of the period-doubling sequence.

\begin{remark}
	We trivially have that, for every $\eps < 1$,
	$$\lim\limits_{\ell \to \infty} \deter_\ell(\omega, \eps) = 0.$$
\end{remark}

\begin{figure}
	\centering
	\begin{subfigure}[b]{0.4\textwidth}
		\includegraphics[scale=0.7]{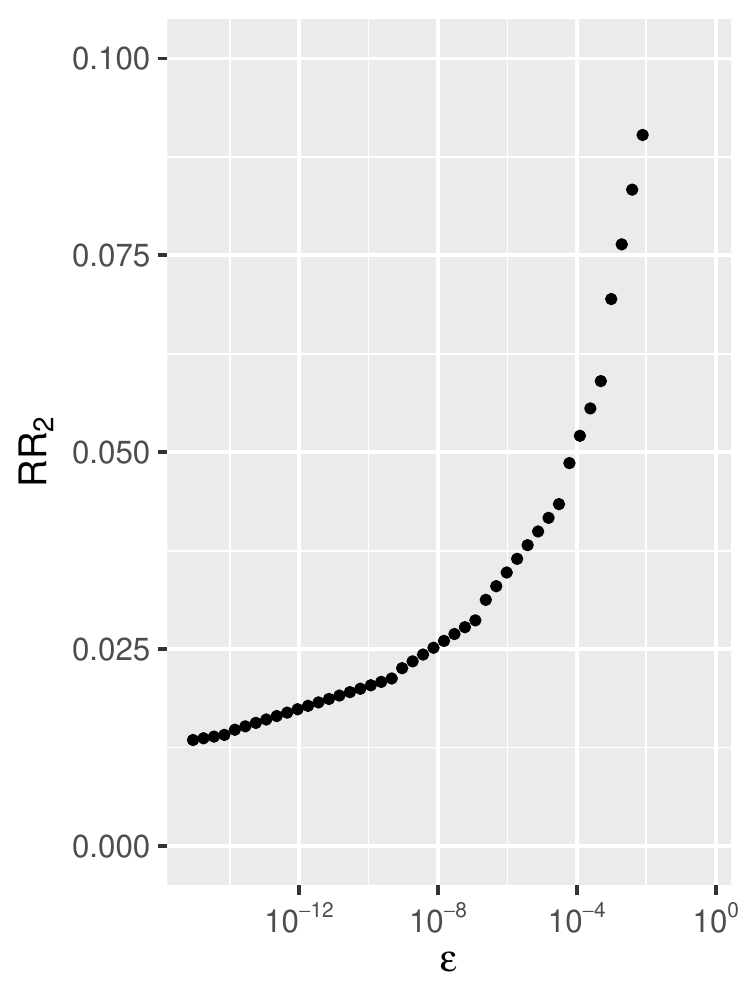}
	\end{subfigure}
	\begin{subfigure}[b]{0.4\textwidth}
		\includegraphics[scale=0.7]{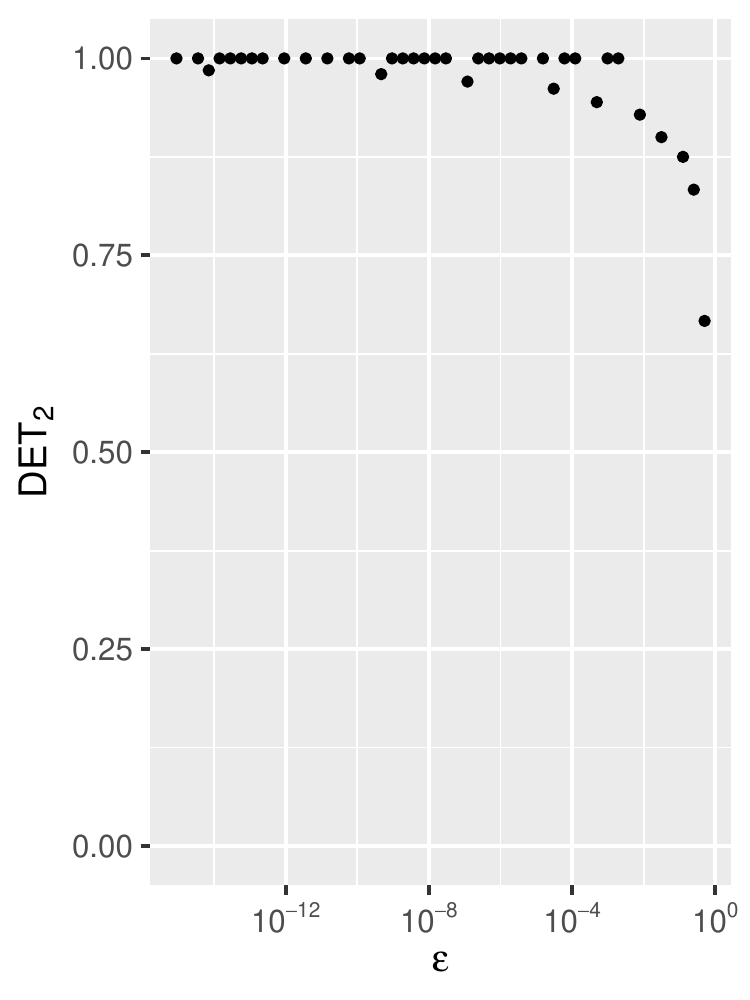}
	\end{subfigure}
\caption{$\rr_2$ and $\deter_2$ of $\omega$.}
\label{FIG:rr2,det2}
\end{figure}

Theorems~\ref{THM:correl-integ-introd}, \ref{THM:RR} and \ref{THM:DET=1} are
stated for embedding dimension $1$. For general embedding dimension, see Subsection~\ref{SUBS:emb-dim}.
See also \cite{vspitalsky2018recurrence} for formulas for other recurrence quantifiers.

This paper is organized as follows.
Preliminaries are given in Section 2.
Complexity of the period-doubling sequence (Theorem~\ref{T:complexity}) is derived in Section 3
as a consequence of some other properties of this sequence.
In Section 4 we give the proof of Theorem~\ref{T:freq-m=general}.
In Section 5 we apply these results to prove Theorems~\ref{THM:correl-integ-introd}, \ref{THM:RR} and \ref{THM:DET=1}.
Moreover, we consider a generalization of our results to arbitrary embedding dimension.

\section{Preliminaries}\label{S:prelim}

The set of positive integers $\{1,2,\dots\}$ is denoted by $\NNN$.
The set $\AAa=\{0,1\}$ is called an \emph{alphabet}.
Put $\AAa^*=\bigcup_{k\ge 0} \AAa^k$; 
$\AAa^*$ endowed with concatenation is a monoid.
Members of $\AAa^*$
are called \emph{words}.
A \emph{word of length $m$}, or an \emph{$m$-word} ($m\ge 1$) is any $v=v_1\dots v_m$ 
from $\AAa^m$ ($m \geq 0$);  $v_i$ is the \emph{$i$-th letter} of $v$.
The \emph{empty word} (the unique word of length $0$) is denoted by $\eps$.
A \emph{subword of $v=v_1\dots v_n$ starting at the $i$-th letter} is any word
$v_i v_{i+1}\dots v_{n'}$ with $i\le n'\le n$.

The \emph{period-doubling} substitution $\zeta$ is defined as follows:
\begin{equation}\label{EQ:subst}
  \zeta:\AAa\to \AAa^*, \qquad
  \zeta(0) = 01,\qquad
  \zeta(1) = 00.
\end{equation}
The substitution $\zeta$ induces a morphism (denoted also by $\zeta$)
of the monoid $\AAa^*$ by putting $\zeta(\eps)=\eps$ and
$\zeta(w) = \zeta (w_1)\zeta(w_2)\dots\zeta(w_n)$ for any nonempty word $w=w_1 w_2 \dots w_n$.
Likewise, $\zeta$ induces a map
(again denoted by $\zeta$) from $\AAa^\NNN$ to $\AAa^\NNN$ by
\begin{equation*}
  \zeta(x)=\zeta(x_1)\zeta(x_2)\dots
  \qquad
  \text{for } x=(x_n)_{n\in\NNN} \in \AAa^\NNN.
\end{equation*}
The iterates $\zeta^k$ ($k\ge 1$) of $\zeta$ are defined inductively by $\zeta^1=\zeta$
and $\zeta^{k}=\zeta\circ\zeta^{k-1}$ for $k\ge 2$.

\emph{Period-doubling sequence} $\uu=0100\,0101\,0100\,0100\,0100\,\dots$ is
the unique fixed point of $\zeta: \AAa^\NNN \to \AAa^\NNN $. Recall that, for every $i\in\NNN$,
$\uu_i$ is equal to $k_i \operatorname{mod} 2$, where $k_i$ is the largest integer such that $2^{k_i}$ divides $i$.
For every integers $m,i\ge 1$, the $m$-word starting
at the position $i$ is denoted by $w_i^{(m)}$:
\begin{equation*}
    w_i^{(m)} = w_i w_{i+1} \dots w_{i+m-1}.
\end{equation*}

For $m=2^k$ ($k\ge 1$) put
\begin{equation*}
  0^{(m)} = \zeta^k(0)
  \qquad\text{and}\qquad
  1^{(m)} = \zeta^k(1)\,;
\end{equation*}
note that both $0^{(m)}$ and $1^{(m)}$ are words of length $m$.

Any subword of $\uu$ (including the empty one) is called \emph{allowed}.
The language $\mathcal{L}_\omega$ of $\uu$ is the set of all allowed words.
The set of all allowed $m$-words is denoted by $\mathcal{L}_\omega^m$.
\emph{Complexity function of $\uu$} is the map $p=p_\uu:\NNN\to\NNN$ such that, for every $m\in\NNN$,
$p(m)=\# \mathcal{L}_\omega^m$ is the number of allowed $m$-words.

Note that for every $m=2^k$ ($k \ge 0$) we have $ 0^{(1)} = 0,  1^{(1)} = 1,$ and
\begin{equation}\label{EQ:0(m)=0(m)1(m), 1=00}
	0^{(2m)} = 0^{(m)}1^{(m)},\qquad 1^{(2m)} = 0^{(m)}0^{(m)}.
\end{equation}

A \emph{measure-theoretical dynamical system} is a system $(X, \mathcal{B}, \mu, f)$,
where $X$ is a set, $\mathcal{B}$ is $\sigma$-algebra over $X$, $\mu$ is a probability measure
and $f:X \to X$ is a $\mu$-measurable and $\mu$-invariant transformation,
i.e.~$f^{-1}(B) \in \mathcal{B}$ and $\mu(f^{-1}(B)) = \mu(B)$ for every $B \in \mathcal{B}$. 
The system $(X, \mathcal{B}, \mu, f)$ is \emph{ergodic} if  $ \mu(B) = 0$ or $\mu(B) = 1$
for every $B \in \mathcal{B}$ with $f^{-1}(B)=B$.

A pair $(X, f)$ is called a \emph{topological dynamical system}
if $X$ is a compact metric space and $f:X \to X$ is a continuous map.
A dynamical system $(X, f)$ is \emph{minimal} if there is no proper subset $M \subsetneq X$
which is nonempty, closed and $f$-invariant (a set $M$ is $f$-invariant if $f(M) \subseteq M$).
Let $\mathcal{B}_X$ denote the system of all Borel subsets of $X$.
A probability measure $\mu$ is said to be \emph{invariant} if $\mu(f^{-1}(A)) = \mu(A)$ for every $A \in \mathcal{B}_X$;
that is, $(X, \mathcal{B}_X, \mu, f)$ is a measure-theoretical dynamical system.
By Krylov-Bogolyubov theorem, for every $(X, f)$ there exists an invariant measure $\mu$.
System $(X, f)$  is called \emph{uniquely ergodic} if such a measure $\mu$ is unique.
Moreover, if $(X, f)$ is also minimal, we call it \emph{strictly ergodic}.

Metric $\rho$ on $\Sigma=\AAa^\NNN$ is defined
for every $\alpha, \beta \in \Sigma$ by $\rho(\alpha, \beta)=0$ if $\alpha = \beta$,
and $\rho(\alpha, \beta)=2^{-k+1}$ if $\alpha \neq \beta$, where $k=\min\{i:\alpha_i \neq \beta_i\}$.
Note that $(\Sigma, \rho)$ is a compact metric space.
For an $m$-word $v$ we define the \emph{cylinder} $[v]$ by
$[v]=\{ \alpha \in \Sigma : \alpha_i = v_i \text{ for } i \leq m\}$.
Cylinders form a basis of the topology and $[v]=B(x,\eps)$
for every $x\in[v]$ and $\eps=2^{-\length{v}}$,
where $B(x,\eps)$ denotes the closed ball with the center $x$ and radius $\eps$.
A \emph{shift} is the map
$\sigma: \Sigma \to \Sigma $ defined by $\sigma(\alpha_1 \alpha_2 \alpha_3 \ldots) = \alpha_2 \alpha_3 \ldots$
For each nonempty closed $\sigma$-invariant subset $Y \subseteq \Sigma$,
the restriction of $(\Sigma, \sigma)$ to $Y$ is called a \emph{subshift}.
The closure of the orbit $(\sigma^n(\alpha))_{n \geq 0}$ of any $\alpha \in \Sigma$ defines a subshift,
as it is always nonempty, closed and $\sigma$-invariant set.
\emph{Period-doubling subshift} is the orbit closure of the period-doubling sequence.

Let $(X, \sigma)$ be a subshift over $A$, $\rho$ be the metric defined above and $\mu$ be a $\sigma$-invariant measure.
\emph{Correlation integral} of $\mu$ is defined for $\eps > 0$ as follows:
$$ \corint(\mu, \eps) =  \mu \times \mu \, \{ (x, y) \in X \times X \colon \rho(x, y) \leq \eps \}.$$
If $2^{-m}\le \eps < 2^{-m+1}$ then clearly
$$
\corint(\mu,\eps) 
= \sum_{v\in\AAa^m} \mu\big([v]\big)^2.
$$

For $x \in X, n \in \NNN$ and $\eps >0$, \emph{correlation sum} is defined by
$$ \corsum(x, n, \eps) = \frac{1}{n^2} \# \left\{ (i, j)\colon 0 \leq i,j <n, \, \, \rho(\sigma^i(x), \sigma^j(x)) \leq \eps \right\}.$$

For uniquely ergodic systems, $\lim_n \corsum(x,n,\eps)=\corint(\mu,\eps)$ for every but countably many $\eps > 0$ and every $x\in X$ \cite{pesin1993rigorous}.

For any $\ell \geq 1$ consider Bowen's metric $$\rho_\ell (\alpha, \beta) =  \max\limits_{0 \leq k< \ell}\rho \left( \sigma^k (\alpha), \sigma^k (\beta) \right). $$
An easy computation gives that we always have
\begin{equation}
\label{EQ:bowen-for-pd}
	\rho_\ell (\alpha, \beta) =
	\begin{cases}
		1 &\text{ if } \alpha_i \neq \beta_i \text{ for some } 1 \leq i \leq \ell, \\
		2^{\ell-1}\rho(\alpha, \beta) &\text{ if }  \alpha_i = \beta_i \text{ for all } 1 \leq i \leq \ell.
	\end{cases}
\end{equation}
We can now define
\begin{equation}
\label{EQ:C_ell}
 \corsum_\ell (x, n, \eps) = \frac{1}{n^2} \# \left\{ (i, j)\colon 0 \leq i,j < n , \, \, \rho_\ell(\sigma^i(x), \sigma^j(x)) \leq \eps \right\}.
\end{equation}

Recurrence quantification analysis (\cite{zbilut1992embeddings}, see also \cite{marwan2007recurrence, webber2015recurrence}) gives several complexity measures quantifying 
structures in recurrence plots, which are useful for visualization of recurrence.
Two of them are \emph{recurrence rate} ($\rr$) and \emph{determinism} ($\deter$).
By \cite[Proposition 1]{grendar2013strong}, recurrence rate and determinism can be expressed by correlation sums as follows:
\begin{equation}
\label{EQ:RR-and-DET}
\rr_\ell = \ell \cdot \corsum_\ell - (\ell-1) \cdot \corsum_{\ell+1} \qquad
\text{ and } \qquad \deter_\ell = \frac{\rr_\ell}{\rr_1},
\end{equation}
where $\ell$ is the minimal required line length; arguments $x, n, \eps$ are omitted
and we consider embedding dimension $1$.
For general embedding dimension $d$ see Subsection~\ref{SUBS:emb-dim}.

If the limit of $\corsum_\ell (x, n, \eps)$ for $n \to \infty$ exists,
it is denoted by $\corsum_\ell (x, \eps)$.
Analogously we define $\rr_\ell (x, \eps)$ and $\deter_\ell (x, \eps)$.

\section{Complexity of the period-doubling sequence}\label{S:complexity}

\subsection{Length $m=2^k$}

In this section, we prove Theorem \ref{T:complexity} in the special case
when the length $m$ is a power of $2$. We start with two lemmas.
The first one follows by induction using \eqref{EQ:0(m)=0(m)1(m), 1=00} and
the second one is a direct consequence of $\zeta^k(\omega) = \omega$.

\begin{lemma}\label{L:0m1m-differs-at-final-letter}
  For any $m=2^k$ ($k\ge 0$), the $m$-words $0^{(m)},1^{(m)}$ differ exactly at the $m$-th letter:
  \begin{equation*}
    (0^{(m)})_i = (1^{(m)})_i \quad \text{for } i<m,
    \qquad
    (0^{(m)})_m \ne (1^{(m)})_m.
  \end{equation*}
  Moreover, if $k$ is even then $(0^{(m)})_m=0$ and $(1^{(m)})_m=1$, and if $k$ is odd then
  $(0^{(m)})_m=1$ and $(1^{(m)})_m=0$.
\end{lemma}

\begin{lemma}\label{L:seq-in-mAlphabet}
  Let $m=2^k$ ($k\ge 0$).
  Then the period-doubling sequence $\uu$ can be written in the form $\uu = (\uu_1)^{(m)} (\uu_2)^{(m)} \dots $.
  That is, for every $i\in\NNN$,
  \begin{equation*}
    w_{(i-1)m+1}^{(m)} =
    \begin{cases}
      0^{(m)}  &\text{if } \uu_i=0,
    \\
      1^{(m)}  &\text{if } \uu_i=1.
    \end{cases}
  \end{equation*}
\end{lemma}

\begin{lemma}\label{L:2words}
  For the period-doubling sequence $\uu$,
  $p(1)=2$ and $p(2)=3$.
  Moreover, the allowed $1$-words are $w_1^{(1)}=0$ and $w_2^{(1)}=1$,
  and the allowed $2$-words are $w_1^{(2)}=01$, $w_2^{(2)}=10$, and $w_3^{(2)}=00$.
\end{lemma}

\begin{proof}
	We only need to prove that the word $11$ is not allowed. But this immediately follows from the fact that $\omega_{2i-1} = 0$ for every $i$.
\end{proof}

\begin{lemma}\label{L:complexity-lower-bound-2k}
  Let $m=2^k$ ($k\ge 1$). Then the words $w_i^{(m)}$ ($1\le i\le \frac{3}{2}m$) are pairwise distinct.
\end{lemma}
\begin{proof}
	We start by showing that, for $\frac{3}{2}m < i \leq 4m$,
\begin{equation}\label{EQ:wi-m}
    w_i^{(m)} =
    \begin{cases}
      w_{i-m/2}^{(m)}   &\text{if } \frac{3}{2}m < i\le 2m;
    \\
      w_{i-2m}^{(m)}   &\text{if } 2m < i\le 3m;
    \\
      w_{i-3m}^{(m)}   &\text{if } 3m < i\le 4m.
    \\
    \end{cases}
\end{equation}
To see this, realize that $\omega = 0^{(m)}1^{(m)}0^{(m)}0^{(m)}\,0^{(m)} \ldots$ by Lemma~\ref{L:seq-in-mAlphabet}.
Hence, by Lemma~\ref{L:0m1m-differs-at-final-letter},
$w_i^{(m)} = w_{i-2m}^{(m)}$ for $2m < i \leq 3m$ and  $w_i^{(m)} = w_{i-3m}^{(m)}$ for $3m < i \leq 4m$.
Furthermore, $\omega = 0^{(n)}1^{(n)}0^{(n)}0^{(n)}\,0^{(n)}1^{(n)} \ldots$, where $n=m/2$.
So analogously, $w_i^{(m)} = w_{i-n}^{(m)}$ for $\frac{3}{2}m < i \leq 2m$.

We now proceed by induction on $k$. For $k=1$, the claim follows from Lemma~\ref{L:2words}.
Assume now that the claim is valid for some $k \geq 1$; we are going to show that it is valid for $k+1$.
Put $m=2^k$. Since $w_i^{(2m)} = w_i^{(m)} w_{i+m}^{(m)}$,
\eqref{EQ:wi-m} and the induction hypothesis yield that
the words $w_i^{(2m)}$ for $1 \leq i \leq 3m$ are pairwise distinct.
\end{proof}

\begin{lemma}\label{L:mwords-in-0m1m}
	Let $m=2^k$ ($k\ge 1$) and $v$ be any allowed $m$-word. Then
	exactly one of the following is true:
	\begin{enumerate}
		\item \label{enum: L:mwords-in-0m1m - caseA}
		$v$ is a subword of $0^{(m)}1^{(m)}$ starting at the $i$-th letter with $i\le m$;
		
		\item \label{enum: L:mwords-in-0m1m - caseB}
		$v$ is a subword of $1^{(m)}0^{(m)}$ starting at the $i$-th letter with $i\le m/2$.
	\end{enumerate}
\end{lemma}

\begin{proof}
	We start by showing that at least one of \eqref{enum: L:mwords-in-0m1m - caseA}, \eqref{enum: L:mwords-in-0m1m - caseB} is true. If $v \in \{ 0^{(m)}, 1^{(m)} \}$, we are done.
	Otherwise, by Lemma~\ref{L:seq-in-mAlphabet},  $v$ is a subword of $0^{(m)}0^{(m)}$ or $0^{(m)}1^{(m)}$ or $1^{(m)}0^{(m)}$,
	starting at an index $j \leq m$.
	By Lemma~\ref{L:0m1m-differs-at-final-letter}, $v$ is a subword of $0^{(m)}1^{(m)}$ or $1^{(m)}0^{(m)}$.
	In the former case we have \eqref{enum: L:mwords-in-0m1m - caseA}.
	In the latter case, we have \eqref{enum: L:mwords-in-0m1m - caseB}
	since $1^{(m)}0^{(m)} = 0^{(n)}0^{(n)}0^{(n)}1^{(n)}$
	by \eqref{EQ:0(m)=0(m)1(m), 1=00}, where $n=m/2$.
	
	Moreover, $\omega$ starts with $0^{(m)}1^{(m)}0^{(m)}$,
	so $v=w_i^{(m)}$ for
	some $1 \leq i \leq \frac{3}{2}m$.
	By Lemma~\ref{L:complexity-lower-bound-2k} the words $w_i^{(m)}$ ($1\le i\le \frac{3}{2}m$)
	are pairwise distinct,
	so only one of \eqref{enum: L:mwords-in-0m1m - caseA} and \eqref{enum: L:mwords-in-0m1m - caseB} is true.
\end{proof}

\begin{proposition} \label{PROP:p(m)-powers-of-2}
	Let $m=2^k$ ($k\ge 1$). Then $ p(m) = \frac{3}{2}m$ and $\mathcal{L}_\omega^{m} = \{ w_i^{(m)}: 1 \leq i \leq \frac{3}{2}m \}$.
\end{proposition}

\begin{proof}
	Lemma~\ref{L:mwords-in-0m1m} gives $ p(m) \leq \frac{3}{2}m$.
	On the other hand, $ p(m) \geq \frac{3}{2}m$ by Lemma~\ref{L:complexity-lower-bound-2k}.
	The description of $\mathcal{L}_\omega^m$ now follows from Lemma~\ref{L:complexity-lower-bound-2k}.
\end{proof}

\begin{remark}
	For $m=2^k \, (k \geq 1)$ we also have $\mathcal{L}_\omega^{m} = \{ w_i^{(m)}: \frac{3}{2}m < i \leq 3m \}$;
	this follows from \eqref{EQ:wi-m}.
\end{remark}

\subsection{General length $m$}

\begin{lemma}\label{L:mwords-general}
  Let $m=2^k+q$, where $k\ge 1$ and $1\le q<2^k$. Let $1\le i<j\le 3\cdot 2^k$. 
  Then $    w_i^{(m)}=w_j^{(m)} $
  if and only if exactly one of the following conditions holds:
  \begin{enumerate}
    \item $1\le i\le 2^k-q$ and $j=i+2^{k+1}$; \label{C1:mwords-general}
    \item $q<2^{k-1}$, $2^k+1\le i\le 3\cdot2^{k-1}-q$, and $j=i+2^{k-1}$. \label{C2:mwords-general}
  \end{enumerate}  
Consequently, for every $1 \leq i < 3 \cdot 2^k$ there is at most one $j$ such that $i < j \leq 3 \cdot 2^k$ and $w_i^{(m)} = w_j^{(m)}$.
\end{lemma}

\begin{proof}
	For $1 \leq i < j \leq 3 \cdot 2^k$ put
	$$ \varphi(i,j) = \min \{ 1 \leq h \leq 2^{k+1}: \, \omega_{i+h-1} \neq \omega_{j+h-1} \}; $$
	it is well-defined by Lemma \ref{L:complexity-lower-bound-2k} applied to the length $2^{k+1}$.
	Note that 
	\begin{equation} \label{EQ: w-eq-iff}
		w_i^{(m)} = w_j^{(m)} \text{ if and only if } \varphi(i, j) > m .
	\end{equation}
	It is clear that
	\begin{equation} \label{EQ: if-phi-rhen}
		\text{if } \varphi(i, j) \geq 2 \quad \text{then } \varphi(i+1, j+1) = \varphi(i, j)-1.
	\end{equation}
	
	Fix $1 \leq i < j \leq 3 \cdot 2^k$ and assume that $w_i^{(m)} = w_j^{(m)}$;
	we are going to show that either \eqref{C1:mwords-general}
	or \eqref{C2:mwords-general} is true.
	Since $m \geq 2^k$ we have that  $w_i^{(2^k)} = w_j^{(2^k)}$
	and, by \eqref{EQ:wi-m}, exactly one of the following is true:
	  \begin{enumerate}[label=(\alph*)]
		\item $1\le i\le 2^k$ and $j=i+2^{k+1}$; \label{C1:proof-mwords-general}
		\item $2^k < i\le 3\cdot2^{k-1}$, and $j=i+2^{k-1}$. \label{C2:proof-mwords-general}
	\end{enumerate}  

	Assume that  \ref{C1:proof-mwords-general} is true.
	Since $\varphi(1, 2^{k+1}+1) = 2^{k+1}$ and $j=i+2^{k+1} \leq 3 \cdot 2^k$, \eqref{EQ: if-phi-rhen}~implies
	\begin{equation} \label{EQ: phi-minus-a}
		\varphi(i, j) = 2^{k+1}-(i-1).
	\end{equation}
	Since $w_i^{(m)} = w_j^{(m)}$ by assumption, \eqref{EQ: w-eq-iff} implies $2^{k+1}-(i-1) > m$,
	that is $i \leq 2^k-q$. So we have \eqref{C1:mwords-general}.
	
	If \ref{C2:proof-mwords-general} is true then,
	by Lemma~\ref{L:seq-in-mAlphabet}, $w_{1+2^k}^{(2^{k+1})} = 0^{(n)}0^{(n)}0^{(n)}1^{(n)}$
	and  $w_{1+2^k+2^{k-1}}^{(2^{k+1})} = 0^{(n)}0^{(n)}1^{(n)}0^{(n)}$ for $n=2^{k-1}$.
	From Lemma~\ref{L:0m1m-differs-at-final-letter} it follows that $\varphi(1+2^k, 1+2^k+2^{k-1}) = 3\cdot 2^{k-1}$.
	Since $2^k < i\le 3\cdot2^{k-1}$ and $j=i+2^{k-1}$, \eqref{EQ: if-phi-rhen} yields
	\begin{equation} \label{EQ: phi-minus-b}
		\varphi(i,j) = \varphi(p+2^k, p+2^k+2^{k-1}) = 3\cdot 2^{k-1}-(p-1),
		\,\, \text{where } p=i-2^k
	\end{equation}
	(notice that $0 < p \leq 2^{k-1}$). By the assumption $w_i^{(m)} = w_j^{(m)}$ and so, by \eqref{EQ: w-eq-iff},
	 $\varphi(i, j) > m = 2^k + q $.
	 Now \eqref{EQ: phi-minus-b} gives $3 \cdot 2^{k-1}-q \geq i$, so we have \eqref{C2:mwords-general}.
	
	Now assume that one of the conditions \eqref{C1:mwords-general}, \eqref{C2:mwords-general} holds.
	If \eqref{C1:mwords-general} holds we have $w_i^{(m)}=w_j^{(m)} $,
	since \eqref{EQ: phi-minus-a} implies $\varphi(i, j) > m$.
	Similarly, if \eqref{C2:mwords-general} is true then $\varphi(i, j) > m $ by \eqref{EQ: phi-minus-b}, so again $w_i^{(m)}=w_j^{(m)} $.

\end{proof}

Now we are ready to prove Theorem~\ref{T:complexity}.

\begin{proof}[Proof of Theorem \ref{T:complexity}]
		It is clear from Lemma~\ref{L:2words} that \eqref{EQ:complexity} is true for $k=0$, so we may assume that $k > 0$.
		Let $n=2^{k+1}$.
		By Proposition~\ref{PROP:p(m)-powers-of-2},
		$$ p(m) = p(n) - \# \{ (i, j):
			1 \leq i < j \leq 3 \cdot 2^k, w_i^{(m)} = w_j^{(m)} \} .$$
		If $q \geq 2^{k-1}$ then only \eqref{C1:mwords-general} from Lemma~\ref{L:mwords-general}
		occurs, consequently, $p(m) = p(n) - (2^k-q) = 4 \cdot 2^{k-1}+q$.
		Otherwise, both \eqref{C1:mwords-general} and \eqref{C2:mwords-general}
		from Lemma~\ref{L:mwords-general} occur and so $p(m) = p(n)-(2^k-q)-(2^{k-1}-q) = 3 \cdot 2^{k-1}+2q$.
\end{proof}

From Theorem~\ref{T:complexity} we immediately have that
$$ p(m+1) - p(m) \in \{1,2\} \quad \text{ for every } m$$
and
$$ \frac{3}{2}m \leq p(m) \leq \frac{5}{3}m \quad \text{ for every } m \geq 2.$$

\section{Invariant measure of the period-doubling subshift}\label{S:invMeasure}

Let $(X, \sigma)$ be the \emph{period-doubling subshift};
i.e.~$X$ is the orbit closure of $\omega$ and $\sigma: X \to X$ is the left shift.
By \cite{michel1976stricte} (see also \cite[Proposition 5.2 and Theorem 5.6]{queffelec2010substitution}),
$(X, \sigma)$ is strictly ergodic.

Denote the unique invariant measure of $(X, \sigma)$ by $\mu$.
By \cite{oxtoby1952ergodic},
$$ \mu \left( [v] \right) = \lim\limits_{n \to \infty} \frac{1}{n} \# \{ 1 \leq i \leq n : w_i^{(m)} = v \}$$
for every $v \in \mathcal{L}^m$.
In this section we prove Theorem~\ref{T:freq-m=general} which gives an explicit formula for measures of cylinders $[v]$.
We follow \cite[Sections 5.3-5.4]{queffelec2010substitution}.
Fix an integer $m \in \NNN$ and recall that $\mathcal{L}^m$ is
the set of all $m$-words in $\omega$.
Define a substitution $\zeta^{(m)}$ over alphabeth $\mathcal{L}^m$ as follows:
for $u \in \mathcal{L}^m$, write $\zeta(u) = y_1 y_2 \ldots y_{2m}$,
and define
$\zeta^{(m)}(u) = (y_1 \ldots y_m)(y_2 \ldots y_{m+1}).$
Let $M^m$ be the composition matrix of $\zeta^{(m)}$,
that is $M^m$ is a $p(m) \times p(m)$ non-negative matrix such that,
for $u, v \in \mathcal{L}^m$, $(M^m)_{uv}$ is
the number of occurencies of $v$ in $\zeta^{(m)}(u)$.
Trivially every member of $M^m$ belongs to $\{0,1,2\}$.

By \cite[Corollary 5.2]{queffelec2010substitution}, the Perron-Frobenius eigenvalue
of $M^m$ is $\lambda = 2$.
Furthermore, if $d^m = (d_u^m)_{u \in \mathcal{L}^m}$ is the unique
normalized eigenvector of $M^m$ corresponding to $\lambda$,
then $\mu([u]) = d_u^m$ by \cite[Corollary 5.4]{queffelec2010substitution}, see also
\cite[Proposition 1]{frid1998frequency}.


\begin{lemma}\label{P:dm-vector-from-queffelec}
  Let $m=2^k\, (k \geq 1)$. Then $d^m = \frac{2}{3m}(1, 1, \ldots, 1)$.
  Consequently, $\mu \left( [v] \right) = \frac{2}{3m}$ for every allowed $m$-word $v$.
\end{lemma}
\begin{proof}
  It is enough to show that every row sum of $M_m$ is equal to $2$.
  For $m=2$ it is easy.
  So assume that $m \geq 4$.
  By \eqref{EQ:wi-m} we have
    \begin{equation}
	  \zeta^{(m)} \left(w_i^{(m)}\right) = w_{2i-1}^{(m)}w_{2i}^{(m)}=
		  \begin{cases}
		  w_{2i-1}^{(m)}w_{2i}^{(m)} &\text{for } i \leq \frac{3}{4}m; \\
		  w_{2i-1-m/2}^{(m)}w_{2i-m/2}^{(m)} &\text{for } \frac{3}{4}m < i \leq m ; \\
		  w_{2i-1-2m}^{(m)}w_{2i-2m}^{(m)} &\text{for } m < i \leq \frac{3}{2}m. 
		  \end{cases}
  \end{equation}
  Hence, for $1 \leq j \leq m$, the word $w_j^{(m)}$ occurs in $\zeta^{(m)}(w_i^{(m)})$
  for $i=\lceil \frac{j}{2} \rceil$ and $i=\lceil\frac{j}{2}\rceil +m $.
  Further, for $m < j \leq \frac{3}{2}m$, the word $w_j^{(m)}$ occurs in $\zeta^{(m)}(w_i^{(m)})$
  for $i=\lceil \frac{j}{2} \rceil$ and $i=\lceil\frac{j}{2}\rceil +\frac{m}{4} $. The proof is complete.

\end{proof}

\begin{proof}[Proof of Theorem~\ref{T:freq-m=general}]
	Theorem~\ref{T:freq-m=general} holds for $q=0$ by the previous lemma, so let $q\geq1$.
	Put $n=2^{k+1}$. If \eqref{C1:T:freq-m=general} is true
	then, by Lemma~\ref{L:mwords-general}, there is exactly one index $j$ such that $i <j \leq 3\cdot 2^k$ and $w_i^{(m)} = w_j^{(m)}$;
	in this case $[v] = [w_i^{(n)}] \sqcup [w_j^{(n)}]$.
	Otherwise, $[v] = [w_i^{(n)}] $.
	Now the theorem follows from Lemma~\ref{P:dm-vector-from-queffelec}.
\end{proof}

\section{Correlation integral and RQA measures}\label{S:correlation-integral}

\begin{proof}[Proof of Theorem \ref{THM:correl-integ-introd} ]
	By \cite{pesin1993rigorous}, modified to uniquely ergodic systems,
	$\lim \corsum(\omega, n, \eps) = \corint(\mu, \eps)$ provided $\corint(\mu, \eps)$ is continuous at $\eps$.
	Since the metric $\rho$ attains only values from $2^{-\NNN_0} \cup \{0\}$,
	$ \corsum(\omega, n, \eps)$ and $\corint(\mu, \eps)$ are
	constant on $\eps \in \left[ 2^{-m}, 2^{-m+1 }\right)$ for every $m$.
	This easily implies $\lim_{n}  \corsum(\omega, n, \eps) = \corint(\mu, \eps)$ for every $\eps$.
	Since
	$$ \corint(\mu, \eps) = \sum_{v\in\mathcal{L}^m}(\mu [v])^2,$$
	Theorem~\ref{T:freq-m=general} and Corollary~\ref{COR:number-of-mu} yield the desired result.
\end{proof}

\begin{corollary}\label{C:ceps-limits}
	Let $0< \eps < \frac{1}{2}$ and $m_\eps$ be defined as in \eqref{EQ:m-eps}. Then
	\begin{equation*}
	\frac{2}{3m_\eps} \leq \corint(\mu, \eps) \leq \frac{25}{36m_\eps}.
	\end{equation*}
	Moreover, if $m_\eps \in \{ 2^k, 2^k+2^{k-1}, k \geq 1 \}$ then $\corint(\mu, \eps) = \frac{2}{3m_\eps}$,
	and if $m_\eps \in \{ 2^k + 2^{k-2}, k \geq 1  \}$ then $\corint(\mu, \eps) = \frac{25}{36m_\eps}$.
	
\end{corollary}

\begin{proof}
	Write $m_\eps=2^k+q$ with $k \geq 1$ and $0 \leq q <2^k$. Let $x = \frac{q}{m_\eps} \in \left[0, \frac{1}{2}\right)$.
	Using Theorem~\ref{THM:correl-integ-introd} and substituting $2^k = m_\eps-q$ into $m_\eps \corint(\mu, \eps) $ we get
	
	\begin{eqnarray*}
		m_\eps \corint(\mu, \eps) = \frac{6-10x}{9(1-x)^2} \quad &\text{for } 0 \leq q < 2^{k-1}, \\
		m_\eps \corint(\mu, \eps) = \frac{5-7x}{9(1-x)^2} \quad &\text{ for } 2^{k-1} \leq q < 2^{k}.
	\end{eqnarray*}
	Using elementary calculus we obtain that $	\frac{2}{3} \leq m_\eps \corint(\mu, \eps) \leq \frac{25}{36}\,$ if $0 \leq q < 2^{k-1}$ and $	\frac{2}{3} \leq m_\eps \corint(\mu, \eps) \leq \frac{49}{72} < \frac{25}{36}\,$ if $2^{k-1} \leq q <2^k$.
	Moreover, minimum is attained at the points $x=0$ and $x=\frac{1}{3}$, corresponding to $q=0$ and $q=2^{k-1}$, and maximum is attained at the point $x=\frac{1}{5}$ corresponding to $q=2^{k-2}$.
\end{proof}

\begin{proof}[Proof of Theorem \ref{THM:RR}]
	If $\eps \geq 1$ then, by \eqref{EQ:RR-and-DET} and Theorem~\ref{THM:correl-integ-introd},
	$\rr_\ell(\omega,n, \eps) = 1$ for every $n$, hence $\rr_\ell(\omega,\eps) = 1$.
	So assume that $0 < \eps < 1$.
	By \eqref{EQ:bowen-for-pd}, for every $x,y \in X$ we have $\rho_\ell(x,y) \leq \eps$ if and only if $ \rho(x,y) \leq 2^{-\ell+1} \eps$.
	So
	$$ \corsum_\ell(\omega,n,\eps) = \corsum(\omega, n, 2^{-\ell+1}\eps).$$
	Thus, by \eqref{EQ:RR-and-DET} and Theorem~\ref{THM:correl-integ-introd},
	\begin{equation} \label{EQ:limRR}
		\rr_\ell(\omega, \epsilon)=\lim\limits_{n \to \infty} \rr_\ell(\omega, n, \eps) = \ell \,\corint(\mu, 2^{-\ell+1}\eps) - (\ell-1)\,\corint(\mu, 2^{-\ell}\eps).
	\end{equation}
	Notice that $m_{2^{-\ell}\eps} = m_\eps+\ell$ and $m_{2^{-\ell+1}\eps} = m_\eps+\ell-1$,
	since $\eps < 1$ and $\ell \geq 1$.
	Put $m=m_\eps+\ell-1$. If $m=1$ (i.e., $m_\eps = 1$ and $\ell = 1$),
	then $\rr_\ell(\omega, \eps) = 5/9$ by \eqref{EQ:limRR} and Theorem~\ref{THM:correl-integ-introd}.
	So we may assume that $m\geq 2$ (i.e. $m_\eps + \ell \geq 3$) and
	hence we may write $m=2^k+q$ with $k \geq 1$ and $0 \leq q < 2^k$.
	
	Now we consider four cases: $q < 2^{k-1}-1$, $q = 2^{k-1}-1$, $2^{k-1} \leq q < 2^k-1$, and $q = 2^k-1$.
	In the first and third cases
	we have $m_{2^{-\ell}\eps} = m+1 = 2^k+(q+1)$ with $q+1 < 2^{k-1}$ and $2^{k-1} \leq q < 2^k$, respectively.
	So \eqref{EQ:limRR} and Theorem~\ref{THM:correl-integ-introd} give the formulas for $\rr_\ell(\omega, \eps)$.
	
	In the second case ($q = 2^{k-1}-1$) we can write $m_{2^{-\ell}\eps} = 2^k + 2^{k-1}$ and
	in the fourth case ($q = 2^k-1$) we can write $m_{2^{-\ell}\eps} = 2^{k+1}+0$; as above,
	\eqref{EQ:limRR} and Theorem~\ref{THM:correl-integ-introd} yield the formula for $\rr_\ell(\omega, \eps)$.

\end{proof}

\begin{proof}[Proof of Theorem \ref{THM:DET=1}]
From \eqref{EQ:RR-and-DET} and the definition of determinism, we have
$$ \deter_\ell(\omega, n, \eps) = \frac{\rr_\ell (\omega, n, \eps)}{\rr_1 (\omega, n, \eps)}. $$
Using \eqref{EQ:limRR} and the fact that $\rr_1 (\omega, \eps) = \corint(\mu, \eps) >0$, we obtain
\begin{equation} \label{EQ:det-pomocou-RR}
	DET_\ell(\omega, \eps) =	\lim\limits_{n \to \infty} DET_\ell(\omega, n, \eps) = \frac{\rr_\ell (\omega, \eps)}{ \rr_1(\omega, \eps)}.\\
\end{equation}

It is clear that $\deter_\ell(\omega, \eps) =1$ for $\eps \geq 1$, so assume that $\eps < 1$. Let $m_\eps = 2^{k} + q$ and $m_\eps + \ell -1 = 2^{k'} + q'$,
where $ k, k' \geq 0,\, 0 \leq q <2^{k}$ and $0 \leq q' <2^{k'}$.
We now compute $\deter_\ell(\omega, \eps)$ using Theorems~\ref{THM:correl-integ-introd}, \ref{THM:RR} and \eqref{EQ:det-pomocou-RR}.
We distinguish three cases.

(a) Let $\eps \in (0,1)$ be such that $k'=k$; then $q' = q+\ell-1$; we write $\eps \in E_a$.
If $0 \leq q < q' <2^{k-1}$ or $2^{k-1} \leq q < q' < 2^k$, we immediately have $ \deter_\ell(\omega, \eps) = 1.$
Otherwise  $0 \leq q <2^{k-1} \leq q' < 2^k$ and
	$$ \deter_\ell(\omega, \eps) = \frac{5\cdot2^k-2q}{6\cdot 2^k-4q} < 1.$$
Here $2^{k-1}-\ell+1 \leq q < 2^{k-1}$ and so $q2^{-k} \to 1/2$ for $\eps \to 0$.
Thus we have 
	$$ \lim_{\substack{\eps \to 0 \\ \eps \in E_a}}  \deter_\ell(\omega, \eps) = 1.$$

(b) Let $\eps \in (0,1)$ be such that $k'=k+1$; we write $\eps \in E_b$.
Then $q' = q+\ell-1-2^k$, and so
$$\deter_\ell(\omega, \eps) = \frac{3\cdot2^k+\Delta}{3\cdot 2^k+2\Delta} < 1, \text{  where  } \Delta = -q'+l-1 \in \{ 1, \ldots, \ell-1 \}.$$
Clearly
$$ \lim_{\substack{\eps \to 0 \\ \eps \in E_b}} \deter_\ell(\omega, \eps) = 1.$$

(c) If  $\eps \in (0,1)\backslash(E_a \cup E_b)$, then $k' \geq k+2$ and we again have $ \deter_\ell(\omega, \eps) < 1$.
Since this can happen only for large enough $\eps$,
this case does not affect the limit $\lim\limits_{\eps \to 0} \deter_\ell (\omega, \eps)$.
(In fact, if $\eps < \min(2^{-(\ell-2)}, 1)\,$ then $m_\eps \geq \ell-1$, and so $2^{k'}+q' = m_\eps + \ell -1 \leq 2m_\eps = 2(2^k+q)$. From this we immediately have $k' \leq k+1$.)
\medskip

Thus we have proved that $\deter_\ell(\omega, \eps) = 1$ if and only if
one of \ref{Case1-in-THM:DET=1}--\ref{Case3-in-THM:DET=1} happens (otherwise  $\deter_\ell(\omega, \eps) < 1$)
and that $\lim\limits_{\eps \to 0} \deter_\ell(\omega, \eps) = 1$.

\end{proof}

\subsection{General embedding dimension} \label{SUBS:emb-dim}
	Up to now we considered recurrence characteristics without embedding.
	The results can be easily generalized to arbitrary embedding dimension $d \geq 1$.
	
	If $x$ is a sequence over $A = \{0,1\}$, then the \emph{embedded sequence} $x^d$ is a sequence over $A ^d = \{0,1\}^d$ defined by
	\begin{equation*}
		x^d = x_1^d  x_2^d \ldots = (x_1 x_2 \ldots x_d)(x_2 x_3 \ldots x_{d+1}) \ldots
	\end{equation*}
	A metric $\rho^d $ in the embedding space $(\AAa^d)^\NNN$ is defined as in Section~\ref{S:prelim}; that is,
	\begin{equation*}
	\rho^d (x^d, y^d) =
	\begin{cases}
	2^{-k+1} &\text{ if } x^d \neq y^d ,  \text{ where }k= \min\{ i: x_i^d \neq y_i^d \}, \\
	0 &\text{ if }  x^d=y^d.
	\end{cases}
	\end{equation*}
	If $k >1$ then trivially
	$$ \rho^d(x^d, y^d) = 2^{-k+1} \qquad \text{ if and only if } \qquad \rho(x,y) = 2^{-(d+k-2)}.$$
	So for correlation sums $\corsum _\ell^d$, defined by~\eqref{EQ:C_ell}
	with $\rho_\ell$ replaced by $\rho_\ell^d$, it holds that
	$$  \corsum _\ell^d (x^d, n, \eps) = \corsum (x, n, 2^{-(l-1)-(d-1)} \eps)$$
	for every $x \in \AAa^\NNN, \eps \in (0,1)$ and $n \in \NNN$.
	This together with Theorem~\ref{THM:correl-integ-introd} yield an explicit formula
	for (embedded) correlation integrals $\corint_\ell^d(\omega^d, \eps)$
	for the period-doubling sequence $\omega$.
	To obtain formulas for $\rr_\ell^d(\omega^d, \eps)$ and $\deter_\ell^d(\omega^d, \eps)$
	it suffices to use \eqref{EQ:RR-and-DET}:
	\begin{eqnarray*}
		\rr_\ell^d (\omega^d, n, \eps)&=& \ell \cdot \corsum_\ell^d(\omega^d,n, \eps) - (\ell-1) \cdot \corsum_{\ell+1}^d(\omega^d, n, \eps), 
		\\
		\qquad \deter_\ell^d(\omega^d, n, \eps)  &=& \frac{\rr_\ell^d(\omega^d, n, \eps)}{\rr_1^d(\omega^d, n, \eps)}.
	\end{eqnarray*}

\section*{Acknowledgements}
The author gratefully acknowledges the many helpful suggestions of Vladim{\'\i}r {\v{S}}pitalsk{\'y}.
This work was supported by the Slovak Research and Development Agency
under the contract No.~APVV-15-0439, and by VEGA grant 1/0768/15.

\bibliography{period-doubling-seq}

\end{document}